\documentclass[12pt]{amsart}
\usepackage[margin=1.2in]{geometry}
\usepackage{graphicx,latexsym}
\usepackage{amsfonts, amssymb, amsmath, amsthm}

\newtheorem{theorem}{Theorem}[section]
\newtheorem{proposition}[theorem]{Proposition}

\newtheorem{corollary}[theorem]{Corollary}

\theoremstyle{definition}
\newtheorem{definition}[theorem]{Definition}
\newtheorem{example}[theorem] {Example}

\theoremstyle{remark}
\newtheorem{remark}[theorem]{Remark}

 \numberwithin{equation}{section}

\newcommand{\RR}{{\mathbb R}}
\newcommand{\MM}{{\mathcal M}}
\newcommand{\NN}{{\mathbb N}}
\newcommand{\CC}{{\mathbb C}}
\newcommand{\LL}{{\mathcal L}}

\begin{document}

\title[Elementary approach to Ces\`{a}ro asymptotics]{An elementary approach to asymptotic behavior  in the Ces\`{a}ro sense and applications to the Laplace and Stieltjes transforms}

\author[D. Nemzer]{Dennis Nemzer}
\address{Department of Mathematics,  California State University, Stanislaus, One University Circle, Turlock, CA  95382, USA}
\email{nemzer@comcast.net}
\thanks{Dedicated to Professor Bogoljub Stankovi\'{c} on the occasion of his 90th birthday.}

\author[J. Vindas]{Jasson Vindas}
\thanks{J. Vindas gratefully acknowledges support by Ghent University, through the BOF-grant 01N01014.}
\address{Department of Mathematics, Ghent University, Krijgslaan 281 Gebouw S22, 9000 Gent, Belgium}
\email{jvindas@cage.UGent.be}


\begin{abstract} We present a elementary approach to asymptotic behavior of generalized functions in the Ces\`aro sense. Our approach is based on Yosida's subspace of Mikusi\'{n}ski operators. Applications to  Laplace and Stieltjes transforms are given.
\end{abstract}

\maketitle

\section{Introduction} 
Asymptotic analysis is a very much studied topic within generalized function theory and has shown to be quite useful for the understanding of structural properties of a generalized function in connection with its local behavior as well as its growth properties at infinity. Several applications have been developed in diverse areas such as Tauberian and Abelian theory for integral transforms, differential equations, number theory, and mathematical physics. There is a vast liteturature on the subject, see the monographs \cite{EstradaKanwal,ML,PilStanTak,PilStanVindas,vdz} and references therein.   

The purpose of this paper is to present an elementary approach to asymptotic behavior in the Ces\`{a}ro sense. The Ces\`{a}ro behavior for Schwartz distributions was introduced by Estrada in \cite{Estrada} (see also \cite{EstradaKanwal}). The approach we develop in this article uses Yosida's algebra $\mathcal{M}$ of operators \cite{yosida}, which provides a simplified but useful version of Mikusi\'{n}ski's operational calculus \cite{Mikusinski}. While Schwartz distribution theory is based on the duality theory of topological vector spaces, the construction of Yosida's space $\mathcal{M}$ is merely algebraic, making only use of elementary notions from calculus. That is why we call our approach to Ces\`{a}ro asymptotics \emph{elementary}.

The plan of the article is as follows. In Section \ref{preliminaries}, we recall the construction of Yosida's space $\mathcal{M}$. We study some useful localization properties of elements of $\mathcal{M}$ in Section \ref{localization}. The asymptotics in the Ces\`{a}ro sense is defined in Section \ref{cesaro} and its properties are investigated. As an application, we conclude the article with some Abelian and Tauberian theorems for Stieltjes and Laplace transforms in Section \ref{applications}. It should be mentioned that Abelian and Tauberian theorems for Stieltjes and Laplace transforms of generalized functions have been extensively investigated by several authors, see, e.g., \cite{Estrada-Vindas,ML,PilStanTak,PilStanVindas,vdz}.
\section{Preliminaries}
\label{preliminaries}

We recall in this section the construction of Yosida's space of operators $\mathcal{M}$ and explain some of its properties. See \cite{yosida} for more details about $\mathcal{M}$.

Let $C^{n}_+(\RR)$  denote the space of all $n$-times continuously differentiable functions on $\RR$ which vanish on the interval $(-\infty , 0)$. We write  $C_+(\RR)=C^{0}_+(\RR)$. For $f, g  \in C_+(\RR)$, the convolution is given by                        
$$(f \ast g)(x) =  \int_0^x f(x-t)g(t) dt.$$                                            

Let $H$ denote the Heaviside function.  That is, $H(x) = 1$ for $x \ge 0$ and zero otherwise.
For each $n \in \NN$, we denote by $H^n$ the function $H \ast \dots \ast H$ where $H$ is repeated $n$ times. One has $H^{n}\in C^{n-2}_{+}(\mathbb{R})$, $n\geq 2$.
Note that if $f \in C_+(\RR)$, then
$$(H^m \ast f)(x) = \frac{1}{(m-1)!} \int_0^x (x-t)^{m-1} f(t) dt.$$

The space $\MM$ is defined as follows,
$$\MM =  \left\{ \frac{f}{H^k} : f \in C_+(\RR),\: k \in \NN \right\}.$$
Two elements of $\MM$ are equal, denoted $\displaystyle \frac{f}{H^n} =  \frac{g}{H^m}$ , if and only if
$H^m \ast f = H^n \ast g$.
Addition, multiplication (using convolution), and scalar multiplication are defined in the natural way, and $\MM$ with these operations is a commutative algebra with identity $ \delta = \displaystyle\frac{H^2}{H^2}$, the Dirac delta. We can embed $C_{+}(\RR)$ into $\mathcal{M}$. Indeed, for $f \in C_+(\RR)$, we set $W_f= \displaystyle\frac{H \ast f}{H}$. Obviously, $f\mapsto W_{f}$ is injective.

Let $W = \displaystyle\frac{f}{H^k} \in \MM$. The generalized derivative of $W$ is defined as 
$
DW =\displaystyle\frac{f}{H^{k+1}}.$
The product of $x$ and $W$ is given by
$$  xW =   \frac{xf-kH \ast f}{H^k} \quad (k \ge 2).
$$ 
In the last formula $k\geq2$ is no restriction because $\displaystyle\frac{f}{H^{k}}=\frac{H\ast f}{H^{k+1}}$. Clearly, these definitions do not depend on the representative of $W$.

\begin{proposition} The generalized derivative and product by $x$ satisfy:
\begin{enumerate}
\item[(a)]  $xW_f = W_{xf}$, $f\in C_{+}(\RR)$.
\item[(b)]  $DW_f = W_{f'},$ $f \in C^1_+(\RR)$.
\item[(c)] $D(xW) = W + xDW, \,\, W \in \MM$.
\end{enumerate}
\end{proposition}
\begin{proof} For (a),
 \begin{align*} 
  xW_f &= \frac{x(H^2 \ast f)}{H^2} - \frac{2H^2 \ast f}{H}
= \frac{xH^2 \ast f}{H^2} + \frac{H^2 \ast xf}{H^2} - \frac{2H^2 \ast f}{H}\\
&=
\frac{2H^3 \ast f}{H^2} + \frac{H^2 \ast xf}{H^2} - \frac{2H^2 \ast f}{H}\\
&
= \frac{H \ast xf}{H} \quad = \quad W_{xf}\:.
\end{align*}
If $f'\in C_{+}(\RR)$, we have $\displaystyle DW_f = \frac{H \ast f}{H^2} = \frac{f}{H} = \frac{H \ast f'}{H} = W_{f'},$ which shows (b). Next, let $\displaystyle W = \frac{f}{H^k} \in \MM$.  Then,
\begin{equation*}
W + xDW = \frac{f}{H^k} + \left( \frac{xf}{H^{k+1}} - \frac{(k+1)f}{H^k} \right)
= \frac{xf}{H^{k+1}} - \frac{kf}{H^k} = D(xW). 
 \end{equation*} 
\end{proof}

\begin{remark}  Notice by identifying  $f \in L^1_{loc}(\RR^+)$ with $\displaystyle\frac{H \ast f}{H} \in \MM$, the space $L^1_{loc}(\RR^+)$  can be considered a subspace of $\MM$. Also, for the construction of $\MM$, the space of locally integrable functions which vanish on $(-\infty , 0)$ could have been used instead of $C_+(\RR)$.
\end{remark}

\section{Localization}
\label{localization}

We discuss in this section  localization properties of elements of $\mathcal{M}$.

\begin{definition} Let $W = \displaystyle \frac{f}{H^k} \in \MM$.  $W$ is said to vanish on an open interval $(a, b)$, denoted $W(x) = 0$ on $(a, b)$, provided there exists a polynomial $p$ with degree at most $k-1$ such that $p(x) = f(x)$ for $a < x < b$. Two elements $W, V \in \MM$ are said to be equal on 
$(a, b)$, denoted $W(x) = V(x)$ on $(a, b)$, provided $W - V$ vanishes on $(a, b)$. 
\end{definition}

The support of $W \in \MM$, denoted supp$\,W$, is the complement of the largest open set on which $W$ vanishes. The degree of a polynomial $p$ will be denoted by $\operatorname*{deg} p$ in the sequel.

\begin{example}  Recall $ \delta = \displaystyle\frac{H^2}{H^2}$.  Notice that $H^2(x) = x$ on the open interval $(0, \infty)$.  Thus, $\delta(x) = 0$ on $(0, \infty)$.  Also, $H^2(x) = 0$ on $(-\infty, 0)$.  So, $\delta(x) = 0$ on $(-\infty, 0)$.  Therefore, $\operatorname*{supp}\delta = \{0\}$.
\end{example}

\begin{proposition}Let $W \in \MM$.
\begin{enumerate}
\item[(a)]  If $W(x) = 0$ on $(a,b)$, then $DW(x) = 0$ on $(a,b)$.
\item[(b)]  If $DW(x)=0$ on $(a,b)$, then $W$ is constant on $(a,b)$.
\end{enumerate}
\end{proposition}
\begin{proof} Part (a) follows immediately from definitions. See \cite[Thm. 4.1]{Nem} for (b). 
\end{proof}

\begin{proposition} \label{WOnInterval}  Let $W \in \MM$.
\begin{enumerate}
\item[(a)]  If $W(x) = 0$ on $(a,b)$, then $xW(x) = 0$ on $(a,b)$.
\item[(b)] Suppose $xW(x) = 0$ on $(a,b)$.  Then, 
\begin{enumerate}
\item[(i)]  $W(x) = 0$ on $(a,b)$,  provided $0 \notin (a,b)$.
\item[(ii)] $W(x) = 0$ on $(a,0) \cup (0,b)$, provided $0 \in (a,b)$.
\end{enumerate}
\end{enumerate}
\end{proposition}
\begin{proof} Let $W =\displaystyle \frac{f}{H^k} \in \MM$ ($k \ge 2$).

\noindent\emph{Part (a).} Since $W(x) = 0$ on $(a,b)$, there exist $a_0, a_1, \dots, a_{k-1} \in \CC$ such that $f(x) = a_0 + a_1x + \dots + a_{k-1}x^{k-1}$ for $a<x<b$.  Now, there exists $A \in \CC$ such that 
$$
xf(x) - k(H \ast f)(x) = A + a_0(1-k)x + a_1\left(1-\frac{k}{2}\right)x^2 + \dots + a_{k-2}\left(1-\frac{k}{k-1}\right)x^{k-1},
$$ 
for $a<x<b$. Since $\displaystyle xW = \frac{xf-kH\ast f}{H^k}$, the above yields $xW(x) = 0$ on $(a,b)$.

\noindent\emph{Part (b)}. Suppose  $xW(x) = 0$ on $(a,b)$.
\begin{enumerate}
\item[(i)]  If $a<0$, then the conclusion is clearly true.  So assume $a \ge 0$.
Now, there exists a polynomial $p$ with $\deg p \le k-1$ such that 
$$xf(x) - k\int_0^xf(t)dt = p(x) \quad \mbox{ for } a<x<b.$$
Thus, $f \in C^1(a,b)$ and $xf'(x) + (1-k)f(x) = p'(x)$ for $a<x<b.$ Solving this differential equation,
it follows that $f(x) = q(x)$ for  $a<x<b,$ where $q$ is a polynomial with $\deg q \le k-1.$ Therefore, $W(x) = 0$ on $(a,b)$. 
\item[(ii)]  Since for all $W \in \MM$, we have $W(x) = 0$ on $(-\infty, 0)$, we only need to show that $W(x) = 0$ on $(0,b)$.  Since $0 \in (a,b)$, one has 
\begin{equation}
\label{eq}
xf(x) - k\int_0^x f(t)dt = 0
\end{equation}
for $a<x<b$, in particular, for $0<x<b$. Similarly as in the proof of part (i), we obtain 
$f(x) = Ax^{k-1}$ for $0<x<b,$ where  $A \in \CC.$
Thus, $W(x) = 0$ on $(0,b)$. Therefore, $W(x) = 0$ on $(a,0) \cup (0,b)$.  
\end{enumerate}
\end{proof}

\begin{proposition}  Let $W \in \MM$.
\begin{enumerate}
\item[(a)] If $xW(x) = 0$ on $(-\infty, \infty)$, then $W = \alpha \delta$ for some $\alpha \in \CC$.
\item[(b)] If $W = \displaystyle \frac{f}{H^k} \,\, (k \ge 2)$ and $W(x) = 0$ on $(0, \infty)$, then $W = \sum_{n=0}^{k-2} \alpha_n \delta^{(n)}$, for some $\alpha_n \in \CC,$ $n = 0, 1, 2, \dots, k-2$.
\end{enumerate}
\end{proposition}
\begin{proof}
Let $W =\displaystyle \frac{f}{H^k} \, \, (k \ge 2)$.  
 
\noindent \emph{Part (a)}. Suppose $xW(x) =0$ on $(-\infty, \infty)$. Similarly as in the proof of Proposition \ref{WOnInterval}(b), we obtain (\ref{eq}) on $(-\infty, \infty)$
It follows that $f(x)= \beta x^{k-1}$ on  $(0, \infty)$, for some $\beta \in \CC.$ Since $f$ is continuous on $\RR$ and vanishes on $(-\infty, 0)$, we have
$f(x) = \beta x^{k-1} \,\, \mbox{ on } [0, \infty).$
Therefore,
$$W = \frac{f}{H^k} = (k-1)! \, \beta \, \frac{H^k}{H^k} = \alpha \delta \quad \mbox{where } \alpha = (k-1)! \beta.$$

\noindent\emph{Part (b)}. Suppose $W(x) = 0$ on $(0,\infty)$.  Then, $f(x) = a_0 + a_1 x + \dots + a_{k-1} x^{k-1}$ on $(0, \infty)$, where $a_0, a_1, \dots, a_{k-1} \in \CC$.  Since $f$ is continuous and vanishes on $(-\infty,0),$ we obtain $a_0 = 0$.  Thus,
\begin{align*}
W & = a_1 \frac{H^2}{H^k} + 2! a_2\, \frac{H^3}{H^k} + \dots + (k-1)!\, a_{k-1}\,\frac{H^k}{H^k}
\\
& = a_1\, \delta^{(k-2)} + 2!\, a_2\, \delta^{(k-1)} + \dots + (k-1)!\, a_{k-1}\, \delta. 
\end{align*}
\end{proof}

\section{Asymptotics in the Ces\`aro sense}
\label{cesaro}
We now introduce and study asymptotics in the Ces\`{a}ro sense for elements of $\mathcal{M}$. As usual, $\Gamma$ stands for the Euler Gamma function and $o$ stands for the little $o$ growth order symbol of Landau \cite[Chap. 1]{EstradaKanwal}.

\begin{definition}\label{def3.1}  Let $W \in \MM$.
For $\alpha \in \RR\backslash\{-1, -2, \dots\}$, define
$$W(x) \sim \frac{\gamma \, x^\alpha}{\Gamma(\alpha + 1)} \quad (C), $$
if and only if there is $k\in\mathbb{N}$ such that $W = \displaystyle \frac{f}{H^k}$ with
\begin{equation}\label{eqC1}
 \frac{\Gamma(\alpha + k +1) f(x)-p(x)}{x^{\alpha + k}} \to \gamma, \quad x \to \infty,
\end{equation}
where $p$ is some polynomial with $\deg p \le k-1$.  That is, 
$$\Gamma(\alpha + k +1) f(x) = p(x) + \gamma x^{\alpha + k} + o(x^{\alpha + k}), \quad \mbox{ as }  x \to \infty.$$
\end{definition}

\begin{remark} If $\alpha >-1$, then the polynomial $p$ is not needed.
\end{remark}

The next theorem tells us that Definition \ref{def3.1} is consistent with the choice of representatives.

\begin{theorem}\label{aclaim}   If $W = \displaystyle \frac{f}{H^k} \in \MM$ is such that (\ref{eqC1}) holds, then the continuous function $H^{m}\ast f$ in the representation $W=\displaystyle\frac{H^m \ast f}{H^{k+m}} \quad (m \in  \NN)$ satisfies
\begin{equation*}\label{eqC2}
 \frac{\Gamma(\alpha + k+m +1) (H^{m}\ast f)(x)-q(x)}{x^{\alpha + k+m}} \to \gamma, \quad x \to \infty,
\end{equation*}
for some polynomial $q$ of degree at most $k+m-1.$
\end{theorem}

\begin{proof} Suppose that (\ref{eqC1}) holds
for some polynomial $p$ with $\deg p \le k-1$. Consider
$$ \Gamma(\alpha +k+2) \int_0^x \left( f(t) - \frac{p(t)}{\Gamma(\alpha + k + 1)} \right)dt + K,$$
with $K$ a constant to be determined.

\noindent
\underline{Case 1.} Suppose $\alpha + k < -1 \Rightarrow x^{\alpha+k} \in L^1(1, \infty)$.
So, (4.1) implies $ f- p/\Gamma(\alpha+k+1)\in L^1(0, \infty)$. Therefore there exists a constant $K$ such that 
$$\Gamma(\alpha + k + 2) \int_0^x \left(f(t) - \frac{p(t)}{\Gamma(\alpha + k + 1)} \right)dt + K \to 0, \quad \mbox{ as } x \to \infty.$$
Thus, using L'Hospital's rule and (\ref{eqC1})

$$\displaystyle\lim_{x \to \infty} \frac{\Gamma(\alpha + k + 2) \int_0^x \left(f(t) - \frac{p(t)}{\Gamma(\alpha + k + 1)} \right)dt + K}{x^{\alpha + k + 1}} 
= \lim_{x \to \infty} \frac{\Gamma(\alpha + k + 1)f(x) - p(x)}{x^{\alpha + k}} = \gamma.$$
Let $q(x) = (\alpha + k+1)\int_0^x p(t)dt - K.$  Then $q$ is a polynomial with $\deg q \le k$. Moreover,
$$\frac{\Gamma(\alpha + k + 2)(H \ast f)(x) - q(x)}{x^{\alpha + k + 1}} \to \gamma, \quad \mbox{ as } x \to \infty.$$
\noindent \underline{Case 2.}  Suppose $\alpha + k > -1.$ Assume $\gamma \ne 0$. Then (\ref{eqC1})  implies
$$ \int_0^x \left(\Gamma(\alpha + k + 1) f(t) - p(t) \right)dt \to \pm \infty, \quad \mbox{ as } x \to \infty.$$
Using L'Hospital's Rule and (\ref{eqC1}),
$$\frac{\Gamma(\alpha + k + 2) \int_0^x\left( f(t) - \frac{p(t)}{\Gamma(\alpha+k + 1)} \right)dt}{x^{\alpha + k + 1}} = \frac{(\alpha + k + 1) \int_0^x (\Gamma(\alpha+k+1)f(t)-p(t))}{x^{\alpha+k+1}} \to \gamma,$$
as  $x \to \infty.$ For the case $\gamma = 0$ (and $\alpha + k > -1)$, we set $g(x) = \Gamma(\alpha + k + 1)f(x) - p(x)$. It is well known that if $x^{-\alpha-k}g(x)\to 0$, then $x^{-\alpha - k - 1} \int_0^x g(t) dt \to 0$ (see e.g. \cite[Chap. 1]{EstradaKanwal}).  

The above shows that the claim is true for $m=1$.  By using induction, the result follows.
\end{proof}

Unless otherwise stated, we assume from now on $\alpha \in \RR\backslash \{-1. -2, \dots\}$.\\

The following provides an alternative to Definition \ref{def3.1}.

Let $W \in \MM$ and $\gamma \in \CC.$ Then

\begin{equation}\label{eqC}
W(x) \sim \frac{\gamma x^\alpha}{\Gamma(\alpha + 1)} \quad (C), \quad x \to \infty
\end{equation}
 if and only if there exist $n \in \NN$ with $\alpha +n >0$, $g \in C_+(\RR)$, and $b>0$ such that $W(x) = D^ng(x)$ on $(b, \infty)$ and $g(x)/x^{\alpha+n}\, \to \,  \gamma/\Gamma(\alpha+n+1)$ as $x \to \infty$. We leave the verification of this fact to the reader.

\begin{theorem} \label{Thm1}
We have:
\begin{enumerate} 
\item[(a)]  If $f \in C_+(\mathbb{R})$ such that $f(x) \sim \gamma x^\alpha$ as $x \to \infty$, then $W_f(x) \sim \gamma x^\alpha \quad (C),$ $x \to \infty$.
\item[(b)]  Let $W \in \MM$ satisfy $(\ref{eqC})$.  Then,
\begin{equation}
\label{eqDW}
DW(x) \sim \frac{\alpha \gamma x^{\alpha -1}}{\Gamma(\alpha+1)} \quad (C), \quad x \to \infty .
\end{equation}
and
\begin{equation}
\label{eqxW}
 xW(x) \sim \frac{\gamma \,  x^{\alpha+1}}{\Gamma(\alpha + 1)}   \quad (C), \quad   x \to \infty.
\end{equation}
\end{enumerate}
\end{theorem}

\begin{proof}  We only prove (\ref{eqxW}) since the other parts of the theorem follow from the definitions.  We first assume a stronger condition on the polynomial $p$. Suppose $\displaystyle V = \frac{g}{H^n} \in \MM \, \, (n \ge 2)$ with 
  \begin{equation} \label{gammaWithp} \frac{\Gamma(\alpha + n + 1) g(x) - p(x)}{x^{\alpha+n}} \to \gamma \quad \mbox{ as } x \to \infty,
  \end{equation}
for some polynomial $p$ with $\deg p \le n-2$. Then there exists a polynomial $q$ with $\deg q \le n-1$ such that 
$$\frac{\Gamma(\alpha+n+2)x g(x) - q(x)}{x^{\alpha+n+1}} \to (\alpha +n+1)\gamma \quad \mbox{ as }\quad x \to \infty.$$
Therefore,
$$\frac{xg}{H^n} \sim \frac{(\alpha + n + 1)\gamma x^{\alpha+1}}{\Gamma(\alpha + 2)} \quad (C),\quad \, \, x \to \infty.$$
Also, from (\ref{gammaWithp})
it follows that
$$\frac{-ng}{H^{n-1}} \sim \frac{-n \gamma x^{\alpha + 1}}{\Gamma(\alpha +2)} \quad (C), \quad \, \, x \to \infty.$$
Thus,
$$xV(x) \sim \frac{\gamma x^{\alpha + 1}}{\Gamma(\alpha+1)} \quad (C), \quad \, \, x \to \infty.$$

We now remove the stronger condition that $\deg p \le n-2$ and complete the proof of the theorem.  Let $W =\displaystyle \frac{f}{H^k} \in \MM$ such that (\ref{eqC}) holds.  That is, there exists $a_0, a_1, \dots, a_{k-1} \in \CC$ such that 
\begin{equation} \label{Existas} \frac{\Gamma(\alpha + k + 1) f(x) - (a_0+a_1 x + \dots + a_{k-1}x^{k-1})}{x^{\alpha+k}} \to \gamma \mbox{ \quad as } x \to \infty \end{equation}
Let $\displaystyle V = \frac{f(x) - \beta x^{k-1}}{H^k} \, \in \MM$, where $\beta = a_{k-1}/\Gamma(\alpha + k + 1)$. Then, from (\ref{Existas}) it follows that,
$$V(x) \sim \frac{\gamma x^\alpha}{\Gamma(\alpha+1)} \quad (C), \quad \, \, x \to \infty.$$
And, by the first part of the proof,
$$xV(x) \sim \frac{\gamma x^{\alpha +1}}{\Gamma(\alpha + 1)} \quad (C), \quad \, \, x \to \infty.$$
Since, $xW = xV + \beta (k-1)! x \delta = xV$, it follows that 
$$ xW(x) \sim \frac{\gamma x^{\alpha+1}}{\Gamma(\alpha+1)} \quad (C), \quad \, \, x \to \infty. $$
\end{proof}

By Theorem \ref{Thm1}, we obtain the following theorem.

\begin{theorem}   If $\displaystyle(H \ast W)(x) \sim \frac{\gamma \, x^\alpha}{\Gamma(\alpha + 1)}  \quad (C),$ then $\displaystyle(xW)(x) \sim \frac{\alpha \gamma \, x^\alpha}{\Gamma(\alpha + 1)}  \quad (C), \\  x \to \infty$.
\end{theorem}

The proofs of the next proposition and corollary follow from the definitions.\\

\begin{proposition}  If $V$ has compact support, then $\displaystyle V(x) \sim \frac{0}{\Gamma(\alpha+1)} x^\alpha \ \ (C),$ $x \to \infty$.
\end{proposition}

Asymptotics in the Ces\`aro sense is a local property.

\begin{corollary}  Let $W, V \in  \MM$.  Suppose that $W$ has Ces\`{a}ro asymptotics $(\ref{eqC})$ and $W(x) = V(x)$ 
 on $(a, \infty)$.  Then,  $\displaystyle V(x) \sim \frac{\gamma \, x^\alpha}{\Gamma(\alpha + 1)}  \quad (C), \quad x \to \infty$.
\end{corollary}

\section{Applications}
\label{applications}
In this last section we give some Abelian and Tauberian theorems for Stieltjes and Laplace transforms of elements of $\mathcal{M}$.

We start by defining the Stieltjes transform \cite{Nem}. Let $r>-1$ and  suppose $\displaystyle W = \frac{f}{H^k} \in \MM$, where $x^{-r-k+\sigma}f(x)$ is {\it bounded as} $x \to \infty$ {\it for some} $\sigma >0$.  The Stieltjes transform of $W$ of index $r$ is given by 
$$\Lambda_r W(z) = (r+1)_k \int_0^\infty \, \frac{f(x)}{(x+z)^{r+k+1}} \, dx \, , \quad  z \in \CC \backslash(-\infty, 0],$$
\noindent where $(r+1)_k = \frac{\Gamma(r+k+1)}{\Gamma(r+1)} = (r+1)(r+2) \dots (r+k).$  Notice that $\Lambda_rW(z)$ is holomorphic in the variable $z$, as one readily verifies.\\

The following is a classical Abelian theorem for the Stieltjes transform.

\begin{theorem}[\cite{Carmichael}] \label{Carm}  If $f \in C_+(\RR)$ such that $x^{-\nu}f(x)\to A$ as $x \to \infty$, with $\nu > -1$, then for $\rho > \nu$,
$$\lim_{\stackrel{z \to \infty}{|\arg z| \le \theta < \pi/2}} \frac{z^{\rho - \nu}\, \Gamma(\rho+1) S_\rho f(z) }{\Gamma(\rho-\nu)\Gamma(\nu + 1)} \, = \, A,$$
\noindent where $S_\rho f(z) = \int_0^\infty \, \frac{f(x)}{(x+z)^{\rho+1}} \, dx.$\\
\end{theorem}

\begin{theorem}  Let $W \in \MM$ and $r > -1$.  Suppose that $\displaystyle W(x) \sim \frac{\gamma x^\alpha}{\Gamma(\alpha+1)} \quad (C), \quad x \to \infty$.  Then:
\begin{enumerate}
\item[(i)] If $r > \alpha > -1$, then $\Lambda_r W(z)$ is well-defined and has asymptotic behavior
$$\lim_{\stackrel{z \to \infty}{|\arg z| \le \theta < \pi/2}} \frac{z^{r-\alpha}\, \Gamma(r+1) \Lambda_r W(z)}{\Gamma(r-\alpha)} \, = \, \gamma,$$
\item[(ii)]  If $\alpha <-1, \, \alpha \notin \{-2, -3, -4, \dots\}$, then $\Lambda_r W(z)$ is well-defined and there are constants $A_1, \dots, A_k$ such that 
\begin{equation} \label{Limitgamma}  \lim_{\stackrel{z \to \infty}{|\arg z| \le \theta < \pi/2}} \frac{z^{r-\alpha}\, \Gamma(r+1)}{\Gamma(r-\alpha)} \left[ \Lambda_r W(z) - \sum_{j=1}^k \frac{A_j}{z^{r+j}} \right] \, = \,  \gamma. \end{equation} \end{enumerate}
\end{theorem}

\begin{proof} 
 (i) Let $\alpha > -1$ and $\displaystyle W = \frac{f}{H^k} \in \MM$ such that 
$$ \frac{\Gamma(\alpha+k+1) f(x)}{x^{\alpha+k}} \to \gamma \quad \mbox{ as } x \to \infty.$$
It follows that for $r > \alpha, f(x)x^{-r-k+\sigma}$ is bounded as $x \to \infty$ for some $ \sigma > 0$. Now, by substituting $r+k$ for $\rho$, $\alpha + k$ for $\nu$, and $\frac{\gamma}{\Gamma(\alpha + k + 1)}$ for $A$ in the above classical Abelian theorem, we obtain
$$\lim_{\stackrel{z \to \infty}{|\arg z| \le \theta < \pi/2}} \frac{z^{r - \alpha} {\Gamma(r+k+1) S_{r+k} f(z)} }{\Gamma(r-\alpha)} \, = \, \gamma.$$
Now, using the fact that $\Lambda_r W(z) = (r+1)_k S_{r+k} f(z)$, the result follows.

(ii) Suppose that $\displaystyle W = \frac{f}{H^k} \in \MM$ with $k+\alpha >-1$ and that $f$ can be written as 
$f(x) = p(x) + g(x),$
where $g \in C_+(\RR)$ satisfies 
$\lim_{x \to \infty}x^{-\alpha - k} g(x) =\gamma/\Gamma(\alpha + k + 1)$
and $p(x) = \sum_{j=0}^{k-1} a_j x^j$.  It follows that $|f(x)| \le Cx^{k-1}$ for some constant $C$ and thus $f(x)x^{-r-k+\sigma}$ is bounded for any $0 < \sigma \le 1+r$. Observe that 
\begin{align*}
S_{r+k}\,f(z) &= S_{r+k}\,g(z) + \sum_{j=0}^{k-1}a_j \int_0^\infty \frac{x^j}{(x+z)^{r+k+1}} \, dx\\
& = S_{r+k}\,g(z) + \sum_{j=0}^{k-1} \frac{j! \Gamma(r+k-j)}{\Gamma(r+k+1)z^{j+k+r}} \, \, a_j.
\end{align*}
By Theorem \ref{Carm}, we have that $z^{r-\alpha}S_{r+k} \, g(z) \to \frac{\gamma \Gamma(r-\alpha)}{\Gamma(r+k+1)}$ as $z \to \infty$ on sectors $|\arg z| \le \theta < \frac{\pi}{2}$.  Since $\Lambda_rW(z) = (r+1)_k S_{r+k}\,f(z)$, we obtain (\ref{Limitgamma}) with $A_j = (k-j)!(r+1)_{j-1} a_{k-j}$.
\end{proof}

We illustrate our ideas with the ensuing example, a deduction of Stirling's formula for the Gamma function.

\begin{example}{\bf (Stirling's formula)}  Recall that the digamma function $\psi$ is defined as the logarithmic derivative of $\Gamma$.  By using the product formula for $\Gamma$, namely,
$$
\Gamma(z)= \frac{e^{- \gamma z}}{z} \prod_{n=1}^{\infty} \left(1+\frac{z}{n}\right)^{-1}e^{\frac{z}{n}},
$$
one has
$$\psi(z) = \frac{\Gamma'(z)}{\Gamma(z)} = -\gamma + \sum_{n=0}^\infty \left( \frac{1}{n+1} - \frac{1}{n+z} \right), \quad z \in \CC \backslash \{0,-1,-2, \dots\},$$
where $\gamma$ is the Euler-Mascheroni constant. We define $\displaystyle W = \frac{f}{H^2}$, where $f(x) = \int_0^x \left( \lfloor t \rfloor - t+ \frac{1}{2}  \right) dt$ (here $\lfloor x \rfloor$ stands for the integer part of $x$). 

Set $g(x) = \lfloor x \rfloor - x + \frac{1}{2}$ and note that $g$ is periodic with period 1, $|g(x)| \le \frac{1}{2}$ for all $x \in \RR$, and $\int_n^{n+1} g(x) dx = 0$ for all $n \in \NN$.  This implies that $\left| \int_0^x g(t) dt \right| \le \frac{1}{2}$ for all $x \ge 0$.  Consequently, $W(x) \sim 0 \cdot x^{-\sigma} \quad (C), \quad x \to \infty$ for any $0 < \sigma < 2$.  Theorem 4 now yields (from the proof it is clear that the constants $A_1 = A_2 = 0$ in this case because $f$ is bounded)
\begin{equation} \label{LimitZero}   \lim_{\stackrel{z \to \infty}{|\arg z| \le \theta < \pi/2}}z^\sigma \Lambda_0 W(z) = 0, \end{equation} 
for any $0<\sigma<2$.  From now on we will work with $1 < \sigma <2$.  We compute an explicit expression for $ \Lambda_0 W(z)$,
\begin{align*}
\Lambda_0 W(z) & = 2 \int_0^\infty \frac{f(x)}{(x+z)^3}\, dx = \int_0^\infty \frac{(\lfloor x \rfloor - x + \frac{1}{2} )}{(x+z)^2} \, dx
=
\frac{1}{2z} + \lim_{N \to \infty} \int_0^N \frac{\lfloor x \rfloor - x}{(x+z)^2} \, dx
\\
&
= \frac{1}{2z} + \lim_{N \to \infty} \left( \int_0^N \frac{d \lfloor x \rfloor}{x+z} - \frac{N}{N+z} + \frac{N}{N+z} - \int_0^N \frac{dx}{x+z} \right)\\
&=\ln z + \frac{1}{2z} + \lim_{N \to \infty} \sum_{n=1}^N \frac{1}{n+z} - \ln N\\
&=
\ln z + \frac{1}{2z} - \psi(z).
\end{align*}
The limit (\ref{LimitZero}) then yields
\begin{equation} \label{psi} \psi(z) = \ln z + \frac{1}{2z} + o \left(\frac{1}{z^\sigma} \right), \quad \, \, z \to \infty, \end{equation} for $z$ in the sectors $|\arg z| \le \theta < \frac{\pi}{2}$.  Note that integration of (\ref{psi}) implies for any $0 < \tau < 1$, 
$$\ln \Gamma(z) = z(\ln z - 1) + \frac{1}{2} \ln z + C + o \left( \frac{1}{z^\tau} \right), \quad z \to \infty$$
on $|\arg z| \le \theta < \frac{\pi}{2}$, which is Stirling's asymptotic formula for the Gamma function except for the evaluation of the constant $C$.  The constant is of course well known to be $C = \sqrt{2\pi}$.  We refer to \cite[p. 43]{EstradaKanwal} for an elementary proof of the latter fact. 
\end{example}
\bigskip

We now consider the Laplace transform \cite{AtanNem}. If $\displaystyle W = \frac{f}{H^k} \in \MM$, where $f(x) e^{-\sigma x}$ {\it is bounded as } $x \to \infty$ {\it for some } $ \sigma \in \RR$, then the Laplace transform of $W$ is given by
$$ \LL W(z) = z^k \int_0^\infty \, e^{-zx} \,f(x) \, dx, \, \, \Re e\: z > \sigma.$$

\begin{theorem}  Let $W  \in \MM$.  Assume that  $\displaystyle W(x) \sim \frac{\gamma x^\alpha}{\Gamma(\alpha + 1)} \quad (C), \, \, x \to \infty$.  Then $W$ is Laplace transformable and
\begin{enumerate}
\item[(i)] If $\alpha > -1$, then 
$$\lim_{\stackrel{z \to 0}{|\arg z| \le \theta < \pi/2}}z^{\alpha + 1} \LL W(z) = \gamma.$$
\item[(ii)]  If $\alpha <-1, \, \alpha \notin \{-2, -3, \dots\}$, then there are constants $A_0, \dots, A_{k-1}$ such that \begin{equation} \label{Laplacegamma} \lim_{\stackrel{z \to 0}{|\arg z| \le \theta < \pi/2}}\, z^{\alpha + 1}\left( \LL W(z) - \sum_{j=0}^{k-1} A_j z^j \right) = \gamma \end{equation} \end{enumerate}
\end{theorem}

\begin{remark}  In the terminology of finite part limits (\cite[Sect. 2.4]{EstradaKanwal}), the limit (\ref{Laplacegamma}) might be rewritten as 
$$\mbox{F.p.} \lim_{\stackrel{z \to 0}{|\arg z| \le \theta < \pi/2}} z^{\alpha + 1} \LL W(z) = \gamma.$$
\end{remark}

\begin{proof} 
(i) Let $\alpha > -1$ and  $\displaystyle W = \frac{f}{H^k} \in \MM$ such that
$$\frac{\Gamma(\alpha+k+1) f(x)}{x^{\alpha+k}} \to \gamma \quad \mbox{ as } x \to \infty.$$
It follows that there exists $\sigma > 0$  such that $f(x)e^{-\sigma x}$ is bounded as  $x \to \infty$. Now, by a well known classical Abelian theorem for the Laplace transform \cite{Doetsch}, we obtain
$$\lim_{\stackrel{z \to 0}{|\arg z| \le \theta < \pi/2}}\frac{z^{\alpha + k + 1} \LL f(z)}{\Gamma(\alpha + k + 1)}  = \frac{\gamma}{\Gamma(\alpha +  k + 1)} \, .$$
So, 
$$\hspace{.75in} \lim_{\stackrel{z \to 0}{|\arg z| \le \theta < \pi/2}}z^{\alpha + 1} \LL W(z) = \lim_{\stackrel{z \to 0}{|\arg z| \le \theta < \pi/2}} z^{\alpha + k + 1} \LL f(z) = \gamma .$$

\noindent (ii)  Suppose $\displaystyle W = \frac{f}{H^k} \in \MM$ with $k + \alpha > -1$ and that $f$ can be written as $f(x) = p(x) + g(x), $
where $g \in C_+(\RR)$ has asymptotic behavior $\lim_{x \to \infty} x^{-\alpha-k}g(x) = \gamma/\Gamma(\alpha+k+1)$ and $p(x) = \sum_{j=0}^{k-1} \alpha_j x^j$.  Note that $W = W_1+W_2$, where $\displaystyle W_1 = \frac{g}{H^k}$ and $\displaystyle W_2 = \frac{p}{H^k}$. Exactly as above, one verifies that 
$$\lim_{\stackrel{z \to 0}{|\arg z| \le \theta < \pi/2}}z^{\alpha + 1} \LL W_1(z) = \gamma.$$
It remains to observe that 
$\LL W_2(z) = \sum_{j=0}^{k-1} A_j z^j,$ with $A_j = \alpha_{k-j-1}/(k-j-1)!$.
\end{proof}

\begin{theorem} \label{Tauberian}  {\bf (Tauberian Theorem)}  Let $\displaystyle W= \frac{f}{H^k} \in \MM$,  where in addition to $f \in C_+(\RR)$, $f$ is real-valued and nonnegative.  If $\LL W(s) \sim \gamma s^{-\alpha - 1}, \, s \to 0^+$ (for some $\alpha > -1$), then $\displaystyle W(x) \sim \frac{\gamma x^\alpha}{\Gamma(\alpha+1)} \quad (C), \, x \to \infty$.
\end{theorem}

\begin{proof}
It follows that
\begin{equation}
 \label{Tauberian1}
   \LL f(s) \sim \gamma s^{-(\alpha + k + 1)}, \quad s \to 0^+. 
		\end{equation}                                                          
Also,
\begin{equation} \label{Tauberian2}
  \LL f(s) = \int_0^\infty e^{-st} d(H \ast f)(t), \quad \, s>0.   
\end{equation}
Since $H \ast f$ is nondecreasing , by (\ref{Tauberian1}) and (\ref{Tauberian2}) and Hardy-Littlewood-Karamata Tauberian Theorem \cite{BingGoldTeug,vdz}, it follows that
$$(H \ast f)(x) \sim \frac{\gamma}{\Gamma(\alpha + k + 2)} \, x^{\alpha + k + 1}, \quad \,\,  x \to \infty.$$
That is,
$$ \frac{\Gamma(\alpha + (k+1) + 1)(H \ast f)(x)}{x^{\alpha+(k+1)}} \to \gamma, \quad \, \, x \to \infty.$$   
Since $\displaystyle W = \frac{H \ast f}{H^{k+1}}$, the above yields $\displaystyle W(x) \sim \frac{\gamma x^\alpha}{\Gamma(\alpha+1)} \quad (C), \, \, x \to \infty.   $     \end{proof}

\end{document}